\documentclass[12pt]{article}
\usepackage[english]{babel}
\usepackage[letterpaper,top=2cm,bottom=2cm,left=3cm,right=3cm,marginparwidth=1.75cm]{geometry}
\usepackage{graphicx}
\usepackage{graphics}
\usepackage{amsfonts}
\usepackage{amscd}
\usepackage{amsthm}
\usepackage{indentfirst}
\usepackage[all]{xy}
\usepackage{amssymb, amsmath, amsthm, amsgen, amstext, amsbsy, amsopn}
\usepackage{epic,eepic}
\usepackage{enumitem} 
\usepackage{color}
\usepackage{ dsfont }
\usepackage[colorlinks=true, allcolors=blue]{hyperref}
\usepackage{verbatim}
\usepackage{subcaption,xcolor}

\theoremstyle{plain}
\newtheorem{thm}{Theorem}[section]
\newtheorem{lemma}[thm]{Lemma}
\newtheorem{prop}[thm]{Proposition}

\newtheorem{cor}[thm]{Corollary}
\newtheorem*{question*}{Question}
\newtheorem*{theorem*}{Theorem}

\theoremstyle{plain}
\newtheorem{thmintro}{Theorem}

\theoremstyle{definition}
\newtheorem{defin}[thm]{Definition}
\newtheorem{rem}[thm]{Remark}
\newtheorem{setup}[thm]{Setup}

\newcommand{\R}{{\mathbb{R}}}

\newcommand{\C}{{\mathbb{C}}}

\newcommand{\Z}{{\mathbb{Z}}}
\newcommand{\N}{{\mathbb{N}}} 
\newcommand{\F}{{\mathbb{F}}}

\newcommand{\cA}{{\mathcal{A}}}

\newcommand{\cT}{\mathcal{T}}
\newcommand{\cN}{\mathcal{N}}

\newcommand{\indCZ}{\operatorname{CZ}}
\newcommand{\indRS}{\operatorname{RS}}
\newcommand{\ind}{\operatorname{ind}}
\newcommand{\Sp}{\operatorname{Sp}}

\title{Periodicity characterization by capacities for star-shaped domains}
\author{Jean Gutt, Vinicius G. B. Ramos and Shira Tanny}

\date{}
\begin{document}

\maketitle
\begin{abstract}
We complete the spectral characterization of Besse and Zoll Reeb flows on the standard contact sphere \(S^{2n-1}\) 
initiated by Ginzburg--G\"urel--Mazzucchelli. Roughly speaking, it states that a Reeb flow on the boundary of any star-shaped domain in $\R^{2n}$ is Besse if and only if it has $n$ coinciding Ekeland--Hofer capacities, and that it is Zoll if and only if the first $n$ capacities coincide. 
\end{abstract}

\section{Introduction and results}
A closed contact manifold $(Y,\alpha)$ is called \emph{Besse} if all its Reeb orbits are closed, and \emph{Zoll} if, in addition, all  orbits have the same minimal period. On $Y=S^{2n-1}\subset \mathbb{R}^{2n}$ with its standard contact structure, contact forms are in bijection with star-shaped domains in $\R^{2n}$. Here and throughout a star-shaped domain is a contractible domain containing the origin, whose boundary is smooth and transverse to the radial vector field\footnote{In other texts, such domains are sometimes called ``strictly star-shaped" or domains whose boundary is a ``restricted contact type hypersurface". }. 

In recent years there have been several works aimed at characterizing and classifying Besse or Zoll flows. In fact, it is an open question whether on $S^{2n-1}$ with the standard structure, there are other Besse contact forms beyond boundaries of ellipsoids and Zoll contact forms beyond the round sphere \cite[p. 2128]{mazzucchelli2023structure}. In dimension 3, classification results for Besse and Zoll contact forms (on the sphere as well as on other contact manifolds) were obtained by Cristofaro-Gardiner and Mazzucchelli  \cite{cristofaro2020action} using Embedded Contact Homology and its capacities. The uniqueness of a Zoll contact form on the 5-sphere was proved recently in \cite{chaidez2025zoll}.  In higher dimensions, a characterization of Besse and Zoll star-shaped domains 
via capacities or spectral invariants was studied by Ginzburg, G\"urel and Mazzucchelli in \cite{ginzburg2021spectral}. For convex domains, they obtained a complete characterization: The flow is periodic if and only if there are $n$ coinciding spectral invariants for the Clarke dual action functional.
For non-convex star-shaped domains they use the Ekeland--Hofer capacities $\{c_k^{EH}\}_{k\in \N}$:
\begin{theorem*}[{\cite[Theorem 1.5]{ginzburg2021spectral}}]
    If $W$ is star-shaped  with discrete action spectrum and $c_{k}^{EH} (W) = \cdots = c_{k+n-1}^{EH}(W)$ for some $k\in \N$ then the Reeb flow on $\partial W$ is Besse and $c_k^{EH}(W)$ is a common period of the flow.
\end{theorem*}
Moreover, they show that $k$ can be obtained from the minimal Conley--Zehnder index  $\mu_-(S_\cT)$  of the orbits of period $\cT=c_k^{EH}(W)$ (see Section~\ref{sec:CZ} for a definition):
$$k=\frac{1}{2}(\mu_-(S_\cT)-n+1).$$
They state the converse direction of the above statement as an open question:
\begin{question*}[{\cite[Question 1.1]{ginzburg2021spectral}}]
    Suppose $W$ is a star-shaped domain with Besse Reeb flow, of common period $\cT$, is it true that 
    \[c_{k}^{EH} (W) = \cdots = c_{k+n-1}^{EH}(W)=\cT \quad \text{for }k=\frac{1}{2}(\mu_-(S_\cT)-n+1)\ ? \]
\end{question*}
As mentioned above, for convex domains this statement was proved in \cite[Theorem 1.2]{ginzburg2021spectral}\footnote{For convex domains the Clarke action spectral invariants coincide with the Ekeland--Hofer capacities by combining \cite{matijevic2024positive} and \cite{gutt2024equivalence}.}. The main purpose of this note is to provide an affirmative answer for a general star-shaped domain. 

We study the characterization of periodicity via equivariant symplectic homology capacities, which we denote by $\{c_k\}_{k=1}^\infty$. These were recently shown to coincide with the Ekeland–Hofer capacities on star-shaped domains \cite{gutt2024equivalence}. Our main results are the following.
\begin{thmintro}\label{thm:Besse}
    Let $W\subset \R^{2n}$ be a Besse star-shaped domain whose Reeb orbits have a common period $\cT$, then there exists $k$ such that 
    \[
    c_k(\alpha)=\cdots = c_{k+n-1}(\alpha) = \cT.
    \]
\end{thmintro}
    In fact, $k=\frac{1}{2}(\mu_-(S_\cT) -n+1)$, as before. This result gives an affirmative answer to Question 1.1 from \cite{ginzburg2021spectral}.
    Theorem~\ref{thm:Besse} follows immediately form the following stronger statement:
\begin{thmintro}\label{thm:Besse_MB}
     Let $W\subset \R^{2n}$ be a Besse star-shaped domain and suppose that the manifold of (parametrized) Reeb orbits of period $T$ is of dimension $2d+1$. Then
    \[
    c_k(\alpha)=\cdots = c_{k+d}(\alpha) = T \quad \text{for} \quad k=\frac{1}{2}(\mu_-(S_T) -n+1).
    \]
\end{thmintro}
This result implies that the capacity-spectrum of a Besse star-shaped domain coincides with the usual spectrum (i.e., the set of periods) and the number of repetitions of each period depends on the dimension of the space of orbits of that period. 
We also show that coincidence of the first $n$ capacities implies the Zoll property.
\begin{thmintro}\label{thm:Zoll}
    Suppose that $W\subset \R^{2n}$ is a star-shaped domain  such that 
    $$c_1(W)=\cdots=c_{n}(W).$$ 
    Then, the Reeb flow on $\partial W$ is Zoll and $c_1(W)$ is the minimal period.
\end{thmintro}

Together with \cite[Theorem 1.5]{ginzburg2021spectral} stated above, Theorems \ref{thm:Besse} and \ref{thm:Zoll} complete the spectral characterization for star-shaped domains.
Finally, our methods also allow us to determine the topological structure of the Morse--Bott families of Reeb orbits on Besse star-shaped domains.
\begin{thmintro}\label{thm:MB_families}
    Let $W\subset \R^{2n}$ be a Besse star-shaped domain and denote by $\varphi^t:\partial W\rightarrow\partial W$ its Reeb flow. For any $T$ in the spectrum of the Reeb flow, the fixed point set 
    \[
    Y_T = \{z:\varphi^Tz =z\} \quad\text{is an $\F_2$-homology sphere, and in particular is connected.} 
    \]
\end{thmintro}
Theorem~\ref{thm:MB_families} is stated with $\F_2$ coefficients for technical reasons. More explicitly, we expect that accounting for orientations in Appendix~\ref{app:MB} would extend Theorem~\ref{thm:MB_families} to $\Z$ coefficients. 

The main new ingredient in the proofs is a lemma showing that Besse star-shaped domains must be dynamically convex. Together with an adaptation of the methods in \cite{cineli2024closed} to the case of Besse Reeb flows, this allows us to understand the symplectic homology barcode of Besse contact forms, and deduce the coincidence of the equivariant SH capacities.

\section*{Acknowledgements}
We are very grateful to Erman \c{C}ineli, Viktor Ginzburg, and Ba\c{s}ak G\"urel  for several useful comments regarding this work. We also thank Yaniv Ganor and Yuan Yao for useful discussions concerning Appendix~\ref{app:MB}. Finally, we are grateful for the hospitality of the IAS where  most discussions took place.

JG was partially supported by the Institute for Advanced Study.

VGBR was partially supported by CNPq productivity grant number 308678/2025-7, FAPERJ Jovem Cientista do Nosso Estado grant number E-26/204.621/2024 and a grant from the Serrapilheira Institute.

ST was partially supported by the Center for New Scientists at the Weizmann Institute of Science, Alon Fellowship and ISF Beresheet grant 4097/25.

\section{Preliminaries}
Since the proofs heavily rely on the results of \cite{cineli2024closed}, we will stick to their notation as much as possible. In particular, this section briefly summarizes the relevant background from sections 2 and 3 in \cite{cineli2024closed}, and we refer there for more details.

\subsection{The Conley--Zehnder index in the Morse--Bott setting} \label{sec:CZ}
\subsubsection{Morse--Bott star-shaped domains}
Let $W$ be a star-shaped domain whose Reeb flow $\varphi^t$ is Besse, namely $\varphi^\cT=\operatorname{id}$ for some $\cT>0$. 
Any Besse contact form is Morse--Bott nondegenerate, and its flow induces an almost-free $S^1$ action on $\partial W$. We denote by $\operatorname{Spec}(\partial W)\subset \R$ the set of periods of periodic orbits. This set is rank-1 (i.e., contained in a rescaling of $\Z$) when the flow is Besse. 

After rescaling  $W$ we may assume that $\varphi^1=\operatorname{id}$. We will work under this normalization assumption from now on.
For any period $T\in \operatorname{Spec}(\partial W)$ of an orbit of $\varphi^t$, the set 
$$Y_T=\{z: \varphi^Tz=z\} \quad \text{is a smooth contact manifold}$$
and $d\varphi^t$ is nondegenerate in the directions transverse to $Y_T$. \begin{rem}[Finiteness]\label{rem:finiteness_Besse}
    For each $T$, $Y_T$ has finitely many connected components, as it is a smooth manifold. Moreover, the Morse--Bott nondegeneracy guarantees that period spectrum is discrete, and therefore there are finitely many periods up to the common period of the flow ($T\leq 1$). For any period $T>1$, the manifold $Y_T$ contains no simple orbits.
    Overall this implies that there are finitely many connected components of manifold of (parametrized) orbits that contain simple orbits. 
\end{rem} 
The quotient of $Y_T$ under the almost-free circle action induced by $\varphi^t$, denoted by 
\[
S_T:= \left\{p\in \partial W: \varphi^T(p)=p\right\}/\{\varphi^t\}_{t\in [0,T]} \quad\text{is a symplectic orbifold.}
\]
We sometimes refer to it as the $T$-critical manifold, but will mostly work with the manifold $Y_T$, rather than the orbifold quotient.
We will often denote its dimension by $2d(S_T)$, or simply $2d$ when the critical manifold is clear from the context. 

\subsubsection{Indices of paths of symplectic matrices}
Let $\widetilde{\Sp}(2n-2)$ be the universal cover of ${\Sp}(2n-2)$. We think of elements of $\widetilde{\Sp}(2n-2)$ as paths in ${\Sp}(2n-2)$ starting at the identity, up to homotopy with fixed ends.

The Robbin--Salamon index of a possibly degenerate path  $\Phi\in \widetilde{\Sp}(2n-2)$ of symplectic matrices is a half-integer
\begin{equation}\label{eq:index_jump_bound}
\indRS(\Phi)\in \frac{1}{2}\Z \quad\text{such that} \quad |\indRS(\Phi)-\indRS(\Phi')|\leq \frac{1}{2} \dim \ker \Phi(1) 
\end{equation}
for any path $\Phi'$ close enough to $\Phi$. For a nondegenerate path it coincides with the Conley--Zehnder index, $\indCZ(\Phi)$.
We define the extremal indices $\mu_\pm(\Phi)$ to be smallest and largest indices of nondegenerate perturbations
\[
\mu_-(\Phi):=\liminf_{\Phi'\rightarrow\Phi}\indRS(\Phi'), \qquad 
\mu_+(\Phi):=\limsup_{\Phi'\rightarrow\Phi}\indRS(\Phi'),
\]
where the limits are taken over nondegenerate perturbations $\Phi'$ of $\Phi$. Then clearly,
\begin{equation}
    \mu_-(\Phi)\leq \indRS(\Phi)\leq \mu_+(\Phi) \quad\text{and}\quad |\mu_{\pm}(\Phi) -\indRS(\Phi)|\leq \frac{1}{2} \dim \ker \Phi(1).
\end{equation}
Actually, $\mu_\pm(\Phi) = \indRS(\Phi) \pm \frac{1}{2} \dim \ker \Phi(1)$. This follows from the index jump bound (\ref{eq:index_jump_bound}), as well as the fact that, for small enough $\epsilon>0$, $\indRS(\{e^{\pm\epsilon J_0 t} \Phi(t)\}_t) = \indRS(\Phi)\pm \frac{1}{2} \dim \ker \Phi(1)$. 
We say that the path $\Phi\in \widetilde{\Sp}(2n-2)$ is dynamically convex if
\[
\mu_-(\Phi)\geq n+1.
\]
The mean index of a path is 
\[
\hat \mu(\Phi):= \lim_{k\rightarrow \infty}\frac{\indRS(\Phi^k)}{k} = \lim_{k\rightarrow \infty}\frac{\mu_\pm(\Phi^k)}{k}, 
\]
where $\Phi^k$ is the composition of $\Phi$ on itself $k$ times. The mean index has two properties that will be useful to us:
\begin{enumerate}
    \item \cite[eq. (3.2)]{cineli2024closed}: For any $\Phi$,
    \begin{equation}\label{eq:mean_index_bound}
        \hat\mu(\Phi)-n+1\leq \mu_-(\Phi)\leq \mu_+(\Phi)\leq \hat\mu(\Phi)
    +n-1.
    \end{equation}
    \item When $\Phi$ is maximally degenerate, all eigenvalues of  $\Phi$ are equal to 1, then
\begin{equation}
    \indRS(\Phi^k) = k\cdot \indRS(\Phi)\quad \text{and hence}\quad \indRS(\Phi) = \hat\mu (\Phi).
\end{equation}
\end{enumerate}
\subsubsection{Indices of closed orbits and Morse--Bott families.}
\begin{lemma}\label{lem:linearized_flows}
    Let $W$ be a star-shaped domain such that the Reeb flow on $\partial W$ is Besse. For any Reeb orbit $x$ on $\partial W$, there is an element 
    $$\Phi\in\widetilde{\Sp}(2n-2)$$
    well-defined up to conjugation by a constant matrix, that  represents the linearized flow along $x$, and:
    \begin{enumerate}
        \item The element corresponding to the $k$-th iterate of $x$ is $\Phi^k$.
        \item \label{itm:Lin_flow_MB_family} If
      $x_0,x_1$ are Reeb orbits lying in the same connected component of a Morse--Bott family, then their linearized flows give the same element:
    \[
    \Phi_0=\Phi_1 \quad\text{in}\quad  \widetilde{\Sp}(2n-2).
    \]
    \end{enumerate}
\end{lemma}
\begin{proof}
   Suppose that $W$ is a star-shaped domain, then all Reeb orbits are contractible and $c_1(\xi)=0$ for $\xi\subset T(\partial W)$ the contact distribution. 
   Given a Reeb orbit $x\subset \partial W$, let $D$ be a disk bounding it, and let 
\[
\tau_D:\xi|_D\rightarrow \R^{2n-2} \quad \text{be a symplectic trivialization of the contact structure.}
\]
We define 
$$\Phi(t):= \tau_D\circ d\varphi^t(x(0))\circ \tau_D^{-1}$$
to be the element of $\widetilde{\Sp}(2n-2)$ corresponding to $x$.
Up to a global change of basis of the target $\R^{2n-2}$ (or, in other words, conjugation by a fixed matrix), this element is independent of the choice of $D$ and $\tau_D$, since any two such disks are homotopic and the corresponding trivializations will induce homotopic paths of matrices with fixed ends. 

   Denote by $S$ the connected component of the manifold $Y_T$ of fixed points of $\varphi^T$, that contains $x_0,x_1$. Consider the sub-bundle 
   $$\xi_S:=\xi|_S\cap TS.$$ 
   The dimension of this sub-bundle is $2d=\dim(S_T)$. The time-$T$ linearized flow satisfies:
   \begin{equation}
       d\varphi^T|_{\xi_S}=\operatorname{id}_{\xi_S}\quad \text{and} \quad (d\varphi^T)^k=\operatorname{id}_\xi \text{ for any $k$ such that }kT\in\N.
   \end{equation} 
   This implies that $d\varphi^T_{x_0(0)}$ is diagonalizable over $\C$:
   \begin{equation}\label{eq:linearized_sum_of_rotations}
   d\varphi^T_{x_0(0)}
   \cong \operatorname{id}_{d,\C} \oplus_{j=1}^{n-1-d} \ e^{i\theta_j}\cdot \operatorname{id}_{1,\C},\quad \text{for }\theta_j\in \left\{\frac{2\pi m}{k}: m=1,\dots, k-1\right\}.
   \end{equation}
   Fix a trivialization $\tau(t):\xi|_{x_0(t)}\rightarrow \R^{2n-2}$ that extends over a disk $D$ bounding $x_0$. 
   After a constant change of coordinates on the target $\R^{2n-2}$, we may assume that
   \[
   \tau(0)\circ d\varphi^T_{x_0(0)}\circ (\tau(0))^{-1}=
   \operatorname{id}_{2d} \oplus_{j=1}^{n-1-d} \ e^{i\theta_j}.
   \]
   Moreover, it follows from the Morse--Bott condition that $\theta_j\neq 0 \mod 2\pi$.
   
    Now, given $x_1$ in the same component $S$, choose a path $x_s(0)$ in $S$ connecting $x_0(0)$ to $x_1(0)$ and set 
    \[(s,t)\mapsto x_s(t):=\varphi^tx_s(0).
    \]
    Extend $\tau(t)$ to a trivialization $\tilde \tau(s,t):\xi|_{x_s(t)}\rightarrow \R^{2n-2}$ such that
    \[
     \tilde\tau(s,0)\circ d\varphi^T\circ (\tilde\tau(s,0))^{-1}= \operatorname{id}_{d,\C} \oplus_{j=1}^{n-1-d} \ e^{i\theta_j^s}\cdot \operatorname{id}_{1,\C}.
    \]
    But, since $\theta_j^s$ is continuous in $s$ and takes values in the finite set $\left\{\frac{2\pi m}{k}: m=1,\dots, k-1\right\}$, then it must be constant in $s$: $\theta_j^s=\theta_j$. 
    We see that under the trivialization $\tilde \tau$ the linearized flows along $x_0,x_1$ are homotopic with fixed ends, so they represent the same element in $\widetilde{\Sp}(2n-2)$.  
    \end{proof}
    
    \begin{cor}
        Let $W$ be a Besse star-shaped domain. For any two Reeb orbits  $x_0,x_1$  lying in the same connected component of a Morse--Bott family $S_T$ of Reeb orbits on $\partial W$, 
    \[
    \indRS(x_0) = \indRS(x_1).
    \]
    \end{cor}

\begin{defin}\label{def:linearized_of_conn_comp}
     In light of Lemma~\ref{lem:linearized_flows}, we define the linearized flow of a connected component $S$ of a Morse--Bott family to be the linearized flow of an orbit from $S$. We will also denote by 
     $$\indRS(S),\quad  \hat\mu(S)\quad \text{and}\quad \mu_\pm(S)$$ 
     the Robbin--Salamon index, mean index and extremal indices  of (any) orbit in $S$.
     Finally, we define $\mu_-(S)$ (respectively, $\mu_+(S)$) to be the minimal (respectively, maximal) indices over all connected components of $S_T$. It holds that (see Appendix~\ref{app:MB}):
     \[
     \mu_\pm(S) = \indRS(S) \pm d,\quad\text{where }d=\frac{1}{2}\dim(S)
     \]
\end{defin}

\subsection{Symplectic homology}\label{subsec:SH}
Given a nondegenerate star-shaped domain $W$, the symplectic homology $SH(W;\mathbb{F})$ is a direct limit of  Hamiltonian Floer homology of
certain admissible Hamiltonians (these are $C^2$–small on $W$ and grow linearly in the Liouville coordinate outside of $W$). The chain complexes for such Hamiltonians have two generators 
$$
\hat x, \check x \quad \text{in degrees}\quad |\hat x| = \indCZ(x)+1, \quad |\check x| = \indCZ(x)
$$ 
for any periodic Reeb orbit $x$ on $\partial W$, and a single critical point $p$ in degree $n$, representing the homology of $W$. The action filtration is given by the action (or Reeb period) for the generators  $\hat x, \check x$ and zero for the critical point $p$.
For any  $t\in \R$, the filtered homology $SH^t(W;\mathbb{F})$ is the homology of the sub-complex of generators with action $<t$. 
Given an interval  $I=(t',t)\subset \R$, $SH^I(W;\F)$ is the homology of the quotient of chain complex $SC^t(W;\F)$ by the image of $SC^{t'}(W;\F)$. 

This filtration induces the structure of a persistence module on symplectic homology. By the structure theorem for persistence modules, $SH(W;\mathbb{F})$ decomposes uniquely as a direct sum of graded interval modules $\mathbb{F}(a,b]$, and the resulting multiset of intervals is the {barcode} 
$$\mathcal{B}(W;\mathbb{F}).$$ 

The {degree} $\deg(I)$ of a bar $I=(a,b]$ is the homological degree in which the corresponding interval module lives.

For every bar $I=(a,b]$, there exist  closed Reeb orbits $x,y$ such that 
\[
\mathcal{A}(x)=a,\qquad \mathcal{A}(y)=b,\qquad 
\deg(I)\in [\mu_-(x),\mu_+(x)+1],\quad \deg(I)\in [\mu_-(y)-1,\mu_+(y)].
\]
In particular this means that $\mu^{-}(y)-\mu^{+}(x)\le 2$ and when equality holds one has $\deg(I)=\mu^{+}(x)+1=\mu^{-}(y)-1$. Bars with $a=0$ start at the extra generator $p$.

\begin{rem}[Smith inequality] \label{rem:smith}
    Let $W$ be a Liouville domain. The Smith inequality (\cite{seidel2015equivariant,shelukhin2022mathbb}, see also \cite[Theorem 2.5]{cineli2024closed}) states that 
    \begin{equation}\label{eq:smith}
        \dim SH^{pI}(W;\mathbb{F}_p) \geq \dim SH^I(W;\mathbb{F}_p).
    \end{equation}
    Together with the universal coefficients theorem this implies that for all $t>0$ and for any field $\mathbb{F}$,
    \[
    \dim SH^t(W;\mathbb{F}) \geq \dim H(W;\mathbb{F}).
    \]
    See {\cite[eq. (4.12)]{cineli2024closed}} for a proof.
\end{rem}
\subsubsection{The equivariant SH capacities}
Below is a review of the equivariant SH capacities, for more details see \cite{gutt2018symplectic}. In this subsection we restrict our attention to fields $\F$ of characteristic zero\footnote{We are grateful to Viktor Ginzburg and Ba\c{s}ak G\"urel for pointing out the importance of characteristic zero for the graded spectrality property.} and omit it from the notation.
\begin{defin}[equivariant SH capacities]
Let $W\subset \R^{2n}$ be star-shaped domain. As a graded vector space, its positive $S^1$–equivariant symplectic homology is 
$$CH_*(W) = \begin{cases}
    \mathbb{F}, & *=n-1+2k,\ k\in\N\\
    0, & \text{otherwise}.
\end{cases}$$ 
It admits an action filtration. The equivariant SH capacities are defined to be
\[
c_k(W) := \inf\left\{L: \operatorname{im}\Big(CH_{n-1+2k}^L(W))\rightarrow CH_{n-1+2k}(W)\Big) \neq 0 \right\}\in(0,\infty]
\]
\end{defin}
The equivariant SH capacities admit several useful properties, but we will only use the (graded) spectrality property.\\

\noindent\textbf{Graded spectrality:}
  If $\alpha=\lambda|_{\partial W}$ is nondegenerate, then
  \[
    c_k(W)=\cA(\gamma)\quad \text{for some Reeb orbit $\gamma$ with}\quad 
    \indCZ(\gamma)=n-1+2k.
  \]
  When $W$ is Morse--Bott, there exists a connected component $S$ of a Morse--Bott family $S_T$ such that
  \[
    c_k(W)=T\quad \text{and}\quad 
    n-1+2k\in [\indRS(S) -d,\indRS(S) +d].
  \]

\section{Besse  domains are dynamically convex}
\begin{lemma}\label{lem:positive_mean_index}
    Suppose that the Reeb flow on $\partial W$ is Besse of minimal common period $1$. Then $\indRS(S_1)>0$.
\end{lemma}
\begin{proof}
    For any Reeb orbit $x$ on $\partial W$, its mean index is a positive multiple of $\hat \mu(S_1) = \indRS(S_1)$. Therefore, if $\indRS(S_1)\leq 0$ then all orbits have non-positive mean indices. By (\ref{eq:mean_index_bound}), this implies
    \[
    \mu_+(x)\leq \hat\mu(x)+n-1\leq n-1.
    \]
    By the definition of $\mu_+$, there are nondegenerate perturbations of $W$ for which all Conley--Zehnder indices of all orbits are bounded by $n-1$. This contradicts the fact that  $c_k(W)$ are finite and satisfy the graded spectrality property. 
\end{proof}
For Besse contact forms, the Robbin--Salamon index of any orbit is an integer. The following  fact from \cite{chaidez2022contact} determines the parity of the index and will be useful.
\begin{lemma}[{\cite[eq. (4.2), (4.3)]{chaidez2022contact}}]\label{lem:pairity_from_CDPT}
    Let $W$ be a Besse star-shaped domain and $S$ a connected component of a Morse--Bott family of dimension $2d$. Then 
    $$\indRS(S) \equiv n-1-d \mod 2\qquad \text{and}\qquad \mu_-(S) \equiv n-1 \mod 2.$$
\end{lemma}
\begin{proof}
    This is a simple consequence of the arguments in \cite[Lemma 4.8]{chaidez2022contact}. More accurately, {\cite[eq. (4.3)]{chaidez2022contact}} states that 
    $\indRS(S) \equiv n-1-d \mod 2$. 
    Moreover, {\cite[eq. (4.2)]{chaidez2022contact}} states that when perturbing a Besse contact form via a Morse function on $S$, the Conley--Zehnder indices of the nondegenerate orbits are equal to 
        $$n-1+\ind_{\operatorname{Morse}}^S(q) \mod 2$$
        where $q$ is the corresponding critical point. In particular, the minimal CZ index is obtained for Morse index 0, and thus
        $$\mu_-(S) \equiv \indRS(S) - d \equiv n-1 \mod 2.\qedhere$$
\end{proof}
Finally, the following lemma is a known phenomenon in Morse--Bott symplectic homology. It follows from \cite{bourgeois2009symplectic}, but for the sake of completeness we include a short proof in the appendix.
\begin{lemma}\label{lem:local_SH_non-zero}
    Let $S_T$ be a  Morse--Bott family of Reeb orbits, then 
    \[SH_{*+\mu_-(S_T)}^{[T-\epsilon,T+\epsilon]}(W;\F_2) = H_*(Y_T;\F_2). 
    \] 
    In particular, 
    $$SH_{\mu_-(S_T)}^{[T-\epsilon,T+\epsilon]}(W;\F_2)\neq 0 \quad\text{and}\quad SH_{\mu_+(S_T)+1}^{[T-\epsilon,T+\epsilon]}(W;\F_2)\neq 0.$$
\end{lemma}

The next proposition is a main new  ingredient in the proof of the main results of this paper.

\begin{prop}\label{prop:Besse_is_dyn_conv}
        Any Besse star-shaped domain is dynamically convex.
\end{prop}
\begin{proof}
    We need to show that for any connected component $S$ of a Morse--Bott family, 
    $$\mu_-(S) \geq n+1.$$
    By Lemma~\ref{lem:positive_mean_index}, $\hat\mu(S_1) = \indRS(S_1)>0$. Since all mean indices of other orbits are positive multiples of $\indRS(S_1)$, they must be positive as well and by (\ref{eq:mean_index_bound}) 
    \[
    \mu_-(x)\geq -n+1 \quad\text{for all Reeb orbits }x.
    \]
    Assume that $W$ is not dynamically convex, then the above lower bound guarantees that there is an orbit of minimal index. Let $k\leq n$ be the minimal value of $\mu_-$ over all Reeb orbits. Among all (connected components of) Morse--Bott families with $\mu_-=k$, let $S\subset S_T$ be the one with the largest action. 
    By Lemma~\ref{lem:local_SH_non-zero}, 
    \[
    SH_k^{[T-\epsilon,T+\epsilon]}(W;\F_2)\neq 0\quad \text{while} \quad SH_k(W;\F_2) =0.
    \]
    This means that there must be a bar $I$ starting at action $T$ with degree $k$, and ending at some action $T'>T$ (alternatively - after a perturbation there must be an orbit coming from a Morse--Bott family $S_{T'}$, ``killing" the generator in degree $k$). Notice  that for any connected component $S'$ of $S_{T'}$,
    \[
    \mu_-(S')>k, \text{since $T$ is the largest action in the minimal degree }k.
    \]
    On the other hand, by Lemma~\ref{lem:pairity_from_CDPT}
    \[k=\mu_-(S) \equiv n-1 \equiv \mu_-(S')\mod 2, \]
    so we must have $\mu_-(S')\geq k+2$, in contradiction to $k=\deg(I) \geq \mu_-(S')-1$ (alternatively, after a perturbation there could not be an orbit from $S_{T'}$ of degree $k+1$).   
\end{proof}
\begin{cor}\label{cor:RS_geq_2n}
    The mean index of the minimal common period Morse--Bott family satisfies 
    $$\hat\mu(S_1) = \indRS(S_1) = \mu_-(S_1)+n-1\geq n+1+n-1= 2n.$$ 
\end{cor}

\section{Symplectic Homology of Besse contact forms}
Our main goal for this section is to prove the following statement.
\begin{prop}\label{prop:SH_1_dim}
     Fix a ground field $\mathbb{F}$. For any Besse contact form on $S^{2n-1}$ and every $t>0$,
     \[
     \dim SH_*^t(W;\mathbb{F})=1.
     \]
\end{prop}
This proposition is an adaptation of \cite[Theorem C]{cineli2024closed} to the case of Besse contact forms, instead of dynamically convex contact forms with finitely many simple orbits. Before discussing the proof let us state a simple corollary that is mentioned in \cite[Theorem E, (O1)]{cineli2024closed}.
\begin{cor}\label{cor:diagonal_barcode}
    For Besse star-shaped domains and $\F=\F_2$, bars with larger period have larger  degree. More accurately, if $I_1,I_2$ are bars such that the action of the beginning of $I_2$ is not smaller that the action of the end point of $I_1$, then $\deg(I_1)<\deg(I_2)$ 
\end{cor}
This section is dedicated to carrying out the proofs of Proposition~\ref{prop:SH_1_dim}, Corollary~\ref{cor:diagonal_barcode} and Theorem~\ref{thm:MB_families}.
The proofs are essentially the same as in \cite{cineli2024closed}, given Proposition~\ref{prop:Besse_is_dyn_conv} which established dynamical convexity for Besse star-shaped domains. Of course, Besse star-shaped domains have infinitely many simple Reeb orbits, but the fact that there are finitely many connected components of Morse--Bott families with simple orbits allows us to adopt the arguments with minimal changes. \\

Let us fix some notations:
\begin{enumerate}
    \item Denote by $X_1,\dots, X_N$ all connected components of Morse--Bott families of orbits containing simple orbits. As explained in Remark~\ref{rem:finiteness_Besse}, there are  finitely many such components. 
    \item We denote by $a_j$ the action of the Morse--Bott component $X_j$. Note that $a_j=a_m$ if $X_j$ and $X_m$ are connected components of the same Morse--Bott family.
    \item Let $\Phi_1,\dots, \Phi_N\in \widetilde{\Sp}(2n-2)$ be the linearized flows of the Morse--Bott components $X_1,\dots, X_N$, as in Definition~\ref{def:linearized_of_conn_comp}.
\end{enumerate}
Notice that 
\[
\hat\mu(\Phi_j) = a_j\cdot \hat \mu(S_1)\quad  \text{and hence the ratio}\quad \frac{a_j}{\hat\mu(X_j)} \quad\text{is constant.}
\]
The following proposition summarizes some of the results of Section 3.2 from \cite{cineli2024closed} in the particular case where the action index ratio is the same for all orbits\footnote{In the language of \cite{cineli2024closed} - there is only one cluster for Besse contact forms.}.
\begin{prop}[Summary of {\cite[Theorem 3.3(IR1,IR5) and Corollary 3.5]{cineli2024closed}}]\label{prop:index_recurrence}
    Let $(\Phi_j, a_j)$ pairs of linearized flows and actions, such that $\hat\mu(\Phi_j) /a_j$ is independent of $j$. For any $\eta>0$
    , there exist sequences $\{C_l\}\subset \R\setminus \operatorname{Spec}(\partial W) $ of bounded gap, $\{d_l\}\subset 2\N$ and $\{k_{jl}\}\subset \N$  all converging to infinity as $l\rightarrow \infty$ such that for all $1\leq j\leq N$
    \begin{itemize}
        \item For all $k>k_{jl}$: $ka_j> C_l$ and $\mu_-(\Phi_j^k)\geq d_l+n+1$.
        \item For all $k<k_{jl}$: $ka_j<C_l-\eta$ and $\mu_+(\Phi_j^k) \leq d_l - 2$.
        \item For $k=k_{jl}$: $ka_j\in(C_l-\eta, C_l)$ and $\mu_+(\Phi_j^k)\leq d_l+n-1$.
    \end{itemize}
\end{prop}

These results are stated in \cite{cineli2024closed} solely in terms of finite collections of elements in $\widetilde{\Sp}(2n-2)\times \R_{>0}$, and therefore we will not repeat the proof from \cite{cineli2024closed}. We will use the above to prove the following:
\begin{prop}[Adaptation of {\cite[Proposition 4.6]{cineli2024closed}}] \label{prop:SH^C_l_dim1}
    Let $W$ be a Besse star-shaped domain. For every field $\mathbb{F}$ there exists a positive real sequence $C_l\rightarrow\infty$ with bounded gap such that
    \[
    \dim SH^{C_l}(W;\mathbb{F}) = 1 \quad \text{for all }C_l.
    \]
\end{prop}
The proof will require the following lemma.
\begin{lemma}[Adaptation of {\cite[Lemma 4.9]{cineli2024closed}}] \label{lem:support_of_SH^C_l}
    Let $C_l$, $d_l$ as in Proposition~\ref{prop:index_recurrence}. There exists $K>0$ such that
    \[
    \dim SH^{C_l}(W) = \dim SH_{d_l+n}^{C_l}(W) \quad \text{for all }C_l>K.
    \]
\end{lemma}
\begin{proof}
    We will show that for any bar $I=(a,b]$ containing $C_l$, $\deg(I)=d_l+n$. Indeed, there exist Reeb orbits $x_j\in X_j$ and $x_{j'}\in X_{j'}$ such that their $k$ and $k'$ iterates satisfy 
    \[
    a=\cA(x_j^k) < C_l< \cA(x_{j'}^{k'}) = b.
    \]
    By Proposition~\ref{prop:index_recurrence}, we must have $k\leq k_{jl}$ and $k'> k_{j'l}$, and thus 
    \[
    \mu_+(x_j^k)\leq d_l+n-1 \quad \text{and}\quad \mu_-(x_{j'}^{k'})\geq d_l+n+1.
    \]
    In particular, the index difference is at least 2. However, the index difference between two Reeb orbits representing the ends of a single bar cannot be larger than 2 (see Section~\ref{subsec:SH} or  \cite[Theorem 2.16]{cineli2024closed}). Therefore the difference must be exactly 2 and:
    \[
    \deg(I) = \mu_-(x_{j'}^{k'})- 1 = \mu_+(x_j^k)+1 = d_l+n.
    \]
    To conclude, any bar $I$ that contains $C_l$ must lie in degree $d_l+n$, which proves the claim.
\end{proof}
\begin{proof}[Proof of Proposition~\ref{prop:SH^C_l_dim1}]
    As explained in equation (2.3) from \cite{cineli2024closed}, the Euler characteristic is independent of $t$:
    \begin{equation}\label{eq:Euler_char}
        \sum_m (-1)^m \dim SH^t_m(W) = (-1)^n \chi(W).
    \end{equation}
    Since $W$ is a star-shaped domain, $\chi(W)=1$. On the other hand, for $t=C_l$, Lemma~\ref{lem:support_of_SH^C_l} states that 
    \[\sum_m (-1)^m \dim SH^t_m(W)  = (-1)^{d_l+n} \dim SH^{C_l}_{d_l+n}(W) = (-1)^n SH^{C_l}_{d_l+n}(W),
    \]
    where the last equality uses the fact that $d_l$ is even (see Proposition~\ref{prop:index_recurrence}). From the Euler characteristic identity (\ref{eq:Euler_char}) it follows that
    \[
    \dim SH^{C_l}(W) = \dim SH^{C_l}_{d_l+n}(W )= 1.\qedhere
    \]
\end{proof}
We are now ready to prove the main result of this section.
\begin{proof}[Proof of Proposition~\ref{prop:SH_1_dim}]
    The proof is exactly the same as that of Theorem C in \cite[p.41]{cineli2024closed}. For the convenience of the reader we sketch it here. 

    By the universal coefficients theorem, it is sufficient to prove this result for the fields $\mathbb{F}_p$ for all primes $p$. Fixing such $p$, denote $D_t:=\dim SH^t(W;\mathbb{F}_p)$, then by Remark~\ref{rem:smith}, $D_t\geq 1$ for all $t>0$. To show the opposite inequality, fix a small interval 
    $$[t,t']\subset \R_{>0}\setminus \operatorname{Spec}(W).$$
    For all $\tau\in [t,t']$, apply (many times) the Smith inequality (\ref{eq:smith}) to the interval $[0,\tau]$. The fact that $p^k\tau$ ranges over the large interval $[p^kt,p^k t']$, and the fact that $\{C_l\}$ has bounded gap, guarantee that 
    \[
    \text{There exist $k$ and $\tau\in[t,t']$ such that $C_l=p^k\tau$, and}\quad D_t = D_\tau\leq  D_{p^k\tau=C_l} = 1.\qedhere
    \]
\end{proof}
\begin{proof}[Proof of Theorem~\ref{thm:MB_families}]
        We need to prove that any critical manifold $Y_T$ of fixed points of $\varphi^T$ (or equivalently, parametrized Reeb orbits), has the homology of a sphere $S^{2d+1}$, and in particular, is connected.
        Suppose there exists $T$ such that $H_*(Y_T)\neq H_*(S^{2d+1})$ as a graded vector space (as before, $d=d(S_T)$ is half the dimension of the $S^1$-quotient). Then, since $Y_T$ is a closed manifold, we conclude that 
        \[
         \dim H_*(Y_T)>2.
        \]
        Recall that by Lemma~\ref{lem:local_SH_non-zero},  $ H_*(Y_T;\F_2) =SH_{*+\mu_-(S_T)}^{[T-\epsilon,T+\epsilon]}(W;\F_2)$ and therefore, there is a bijection between the elements of $H_*(Y_T)$ and the bars intersecting exactly one of the ends of interval $[T-\epsilon, T+\epsilon]$ (but not both of them). On the other hand, by Proposition~\ref{prop:SH_1_dim} for $\F=\F_2$, $\dim SH_*^t(W;\F_2)=1$ for all $t$, and therefore exactly one bar must enter and one must leave the action window $[T-\epsilon, T+\epsilon]$, and therefore the number of bars intersecting the ends cannot be grater than  2.
\end{proof}

We conclude this section with a proof of Corollary \ref{cor:diagonal_barcode}, which follows the arguments from \cite[Section 4.2]{cineli2024closed}.
\begin{proof}[Proof of Corollary \ref{cor:diagonal_barcode}]
    As in the proof of Theorem~\ref{thm:MB_families} above,  Lemma~\ref{lem:local_SH_non-zero} gives a bijection between the two elements of $H_*(Y_{T};\F_2)$, lying in degrees $0,2d+1$, and classes in $SH^{[T-\epsilon,T+\epsilon]}(W;\F_2)$, which lie in degrees 
    $$\mu_-(S_{T})\quad\text{and}\quad\mu_+(T)+1.$$ 
    These classes are  in bijection with bars intersecting one of the ends of the interval $[T-\epsilon,T+\epsilon]$. Together with Proposition~\ref{prop:SH_1_dim}, this implies that there exists one bar $I^+$ starting at action $T$ and one bar $I^-$ ending at action $T$. Their degrees are either
     \begin{equation}\label{eq:bar_degrees_1}
        \deg(I^+) = \mu_+(S_{T})+1 \quad\text{and}\quad \deg(I^-)=\mu_-(S_{T})-1,
    \end{equation}
    or
     \begin{equation}\label{eq:bar_degrees_2}
        \deg(I^-) = \mu_+(S_{T}) \quad\text{and}\quad \deg(I^+)=\mu_-(S_{T}),
    \end{equation}
    depending which generator is an end, and which one is a beginning.
    Recall that Lemma~\ref{lem:pairity_from_CDPT} dictated the parity of $\mu_-(S_T)$:
    \[
     \mu_-(S_T)\equiv n-1 \mod 2\quad\text{and}\quad \mu_+(S_{T}) =  \mu_-(S_{T})+2d\equiv n-1 \mod 2.
    \]
    So, either (\ref{eq:bar_degrees_1}) holds and $\deg(I^\pm)\equiv n\mod 2$, or (\ref{eq:bar_degrees_2}) holds and $\deg(I^\pm)\equiv n-1\mod 2$.
    We will show, by induction on the action, that only the first option is possible.
    Let $\{T_i\}$ be the elements in the action spectrum in an increasing order. For $i=1$, $I_1^-$ should start at the class $p$ representing $W$, which lies in degree $n$ (and action 0). Therefore $\deg(I_1^-) =n$, which means that (\ref{eq:bar_degrees_1}) holds and we must have $\deg(I_1^+) \equiv n \mod 2$. Since $I_1^+=I_2^-$, we see that (\ref{eq:bar_degrees_1}) holds for $i=2$ as well, and we continue by induction.  
    Overall, we conclude that (\ref{eq:bar_degrees_1}) holds for all $T_i$, and since $I_i^-=I_{i-1}^+$, it implies that 
    \[
    \deg(I_i^+) = \mu_+(S_{T_i})+1>\mu_-(S_{T_i})-1=\deg(I_i^-) = \deg(I_{i-1}^+).\qedhere
    \]
\end{proof}

\section{Proofs of the main Theorems}
In this section we prove Theorems~\ref{thm:Besse} and \ref{thm:Zoll}. The following consequence of Proposition~\ref{prop:SH_1_dim} and Corollary~\ref{cor:diagonal_barcode} will be useful and is possibly of independent interest:
\begin{cor}\label{cor:no_index_overlap}
  Let $W$ as above. For all $T_1<T_2$ in $\operatorname{Spec}(\partial W)$, the interval 
    \begin{equation}\label{eq:index_overlap}
        [\mu_-(S_{T_1}),\mu_+(S_{T_1})] \quad\text{does not intersect}\quad [\mu_-(S_{T_2}),\mu_+(S_{T_2})].
    \end{equation}
\end{cor}
\begin{proof}
    We will show that $\mu_+(S_{T_1})< \mu_-(S_{T_2})$. Applying Lemma~\ref{lem:local_SH_non-zero} to $S_{T_i}$, we see that 
    \[
    SH_{*+\mu_-(S_{T_i})}^{[1-\epsilon,1+\epsilon]}(W;\F_2) = H_*(S_{T_i};\F_2)\neq 0 \quad\text{for}\quad *=0, 2d+1=\dim(Y_{T_i}). 
    \]
    In particular, there are bars $I_\pm^i$ starting and ending at action $T_i$, of degrees:
    \[
    \deg(I_-^i) = \mu_-(S_{T_i})-1\quad \text{and}\quad \deg(I_+^i) = \mu_+(S_{T_i})+1.
    \]    
    Since degrees of starting bars strictly increase with the action values (Corollary~\ref{cor:diagonal_barcode}) and $T<1$, we have, 
    \[
    \mu_+(S_{T_1})+1 = \deg(I_+^1) \leq\deg (I_-^2) = \mu_-(S_{T_2}) - 1. 
    \]
    In particular, $\mu_+(S_{T_1})<\mu_-(S_{T_2})$ as required.
\end{proof}

\begin{proof}[Proof of Theorem~\ref{thm:Besse_MB}]
    Let $S_T$ be a Morse-Bott family of action $T$ and dimension $2d$
    By Lemma~\ref{lem:pairity_from_CDPT}, $\mu_(S_T)\equiv n-1 \mod 2$. We may write 
    $$
    \mu_-(S_T)=n-1+ 2k \quad \text{for some}\quad  k\in \N.
    $$
    In this case,  $\mu_+(S_T)=n-1 +2(k+d)$. We claim that 
    \[
    c_k(W)=\cdots=c_{k+d}(W)=T.
    \]
    Indeed, by the graded spectrality property of $c_k$, it is enough to show that there are no generators of action $T'\neq T$ in the relevant degrees: $n-1+2k, \cdots, n-1+2(k+d)$. But this follows from Corollary~\ref{cor:no_index_overlap}.
\end{proof}
\begin{proof}[Proof of Theorem~\ref{thm:Besse}]
    Apply Theorem~\ref{thm:Besse_MB} to the top dimensional family $S_1$, with $d=n-1.$
\end{proof}

The proof of Theorem~\ref{thm:Zoll} requires the following lemma.
\begin{lemma}\label{lem:c_1=T}
    If $W$ is Besse of minimal common period $1$, and  $c_1(W)\in \N$ is a common period, then $\hat\mu(S_1)=2n$ and $c_1(W)=1$.
\end{lemma}
\begin{proof}
    By Corollary~\ref{cor:RS_geq_2n} $\hat \mu(S_1)\geq 2n$. If $\hat \mu(S_1)>2n$ then for every $k$, $\hat\mu(S_{k})= k\cdot \hat\mu(S_1)>2n$, and 
    $$\mu_-(S_{k})=\indRS (S_{k})- n+1 = \hat\mu(S_{k})-n+1>n+1,$$
    for all $k\in \N$. Hence, there are no generators of action in $\N$ in grading $n+1$, which contradicts $c_1(W)\in \N$.

    To show that $c_1(W)=1$, we notice that 
    $$\mu_-(S_{k})\geq \hat\mu(S_{k})-n+1 = k\cdot \hat\mu(S_1) - n+1 = 2nk-n+1 > n+1$$
    whenever $k>1$, so there are no generators of action $k$ in grading $n+1$.
\end{proof}
\begin{proof}[Proof of Theorem~\ref{thm:Zoll}]
    Let $W$ be a Besse star-shaped domain of minimal common period 1 and assume 
    \[
    c_1(W)=\cdots =c_n(W).
    \]
    Theorem 1.5 from \cite{ginzburg2021spectral} states that 
    \[\text{if $c_i(W)=\cdots=c_{i+n-1}(W)$ then $W$ is Besse}\]
    and $c_i(W)$ is a common period of its closed Reeb orbits. Therefore, in our case we conclude that $W$ is Besse and that $c_1(W)$ is a common period for its closed Reeb orbits. In particular,   
    $c_1(W)\in \N$. By Lemma~\ref{lem:c_1=T}, $c_1(W)=1$ and $\hat\mu(S_1)=2n$. Suppose for the sake of contradiction that $W$ is not Zoll. Then there exists Morse--Bott family $S_T$ such that 
    \[
    T\leq 1/2\quad\text{and}\quad \hat\mu(S_T) = T\cdot \hat\mu(S_1) \leq \frac{1}{2}\cdot 2n = n.
    \]
    In particular $\mu_-(S_T)\leq \hat\mu(S_T)+n-1 = 2n-1$. On the other hand, by Proposition~\ref{prop:Besse_is_dyn_conv}, Besse star-shaped domains are dynamincally convex and hence $\mu_-(S_T)\geq n+1$. Overall,
    \[
    \mu_-(S_T)\in [n+1, 2n-1]\subset [n+1, 2n+n-1] = [\mu_-(S_1), \mu_+(S_1)]
    \]
    which contradicts Corollary~\ref{cor:no_index_overlap}.
\end{proof}

\newpage
\appendix
\section{SH of a Morse--Bott family}\label{app:MB}
This section is dedicated to a sketch of a proof of Lemma~\ref{lem:local_SH_non-zero}. It is a special (and simpler) case of arguments from \cite{bourgeois2009symplectic}. 
\begin{lemma}[Lemma \ref{lem:local_SH_non-zero}]
    Let $W\subset \R^{2n}$ be a star-shaped domain whose Reeb flow is Besse, and let $S_T$ be a  Morse--Bott family of Reeb orbits, then 
    \[SH_{*+\mu_-(S_T)}^{[T-\epsilon,T+\epsilon]}(W;\F_2) = H_*(Y_T;\F_2). 
    \] 
    In particular, 
    $$SH_{\mu_-(S_T)}^{[T-\epsilon,T+\epsilon]}(W;\F_2)\neq 0 \quad\text{and}\quad SH_{\mu_+(S_T)+1}^{[T-\epsilon,T+\epsilon]}(W;\F_2)\neq 0.$$
\end{lemma}

In what follows, $R$ is the Reeb vector field on $\partial W$ and $\varphi^t$ is the Besse Reeb flow. We assume without loss of generality that the minimal common period is 1, $\varphi^1=\operatorname{id}$. In particular, the flow is Morse--Bott. Fix a period $T$. As before, we denote by $Y_T\subset \partial W$ the image of this Morse--Bott family (for example, if $T=1$ then $Y_T=\partial W$). 
 We will show that for a certain nondegenerate Hamiltonian, the Floer chain complex in a small action window around $T$ is isomorphic to the Morse complex of some Morse function on $Y_T$. The claim for the symplectic homology chain complex is obtained by a standard direct limit argument.

We collect here some assumptions and notations that will be of use.
\begin{setup}\label{setup:J}
    \begin{enumerate}
        \item Let $\rho$ be the Liouville coordinate, so 
        \[
        \R^{2n}\setminus \{0\}\ \cong \ (0,\infty)\times \partial W,\qquad \rho\cdot y\mapsto (\rho,y).
        \]      
        \item Let $I\subset (0,\infty)$ be a small open interval containing 1. We extend the flow $\varphi^t$ to $I\times \partial W$, so that in forms an $S^1$-action on that neighborhood.
        \item Let $J$ be  an $\omega_0$-compatible almost complex structure on $\R^{2n}$ that, near $\partial W$, is invariant under the $S^1$-action generated by the flow\footnote{Such an almost complex structure can be constructed by averaging an arbitrary one  under the $S^1$-action.}, and let $g_J$ be the Riemannian metric that is compatible with $(\omega_0, J)$.
        \item Fix a tubular neighborhood $\widehat \cN$ of $I\times Y_T$ in $\R^{2n}$ that is isomorphic to a $g_J$-normal disk bundle of $I\times Y_T$. As a consequence: 
        \begin{itemize}
            \item There is a projection $\pi:\widehat \cN\rightarrow I\times Y_T$, coming from the normal bundle structure.
            \item The tangent bundle of $\widehat \cN$ is isomorphic to the tangent bundle of the normal bundle and therefore splits as
        \[
        T_z\widehat \cN \cong V_N\oplus T(I\times Y_T) =: V_N\oplus V_T,
        \]
        where $V_N=\ker d\pi$ is isomorphic to the $g_J$-normal sub-bundle to $Y_T$ at $\pi(z)$. 
        \end{itemize}       
        \item Since the decomposition $T\widehat \cN \cong V_N\oplus V_T$ is $g_J$-orthogonal and $V_T$ is symplectic (as $Y_T\subset \partial W$ is a contact submanifold), then $J$ respects the splitting: $JV_T=V_T$, $JV_N=V_N$.
    \end{enumerate}
\end{setup}

Let $f:Y_T\rightarrow \R$ be a Morse function on $Y_T$ and extend it to a Morse function of $I\times Y_T$ by adding a  function of the form $(\rho-1)^2$. So the critical points still lie on $Y_T$ and $\frac{d^2}{d\rho^2}f>0$. If necessary, we perturb this Morse function to guarantee that $(f,g_J|_{Y_T})$ is Morse--Smale\footnote{This is possible due to the Kupka-Smale Theorem, cf. \cite[Theorem 6.6]{banyaga2004lectures}.}. Next, we extend this function to $\widehat\cN$ by composing with the projection $\pi$, and still denote it by $f$. So:
\[
df|_{V_N} = df|_{\ker d\pi}=0.
\]

Let $\beta:\R^{2n}\rightarrow \R$ be a bump function supported in $\widehat\cN$ that equals 1 on some smaller neighborhood $\cN\Subset \widehat\cN$ of $Y_T$. For any $\delta>0$ we define a time-dependent Hamiltonian by
\[
H_{\delta}(t,z) := h(\rho)+ \delta\cdot \beta(z)\cdot f\big(\varphi^{-T\cdot t}(y)\big),\quad \text{where}\quad z=(\rho,y),
\]
where $h(\rho)\equiv T\cdot \rho$ on $\cN$, and is strictly smaller (respectively, bigger) for smaller (respectively, bigger) $\rho\notin I$ .
The Hamiltonian vector field of this perturbation is
\[
X_{H_{\delta}} = T\cdot R + \delta\cdot  f\big(\varphi^{-T\cdot t}(y)\big)\cdot  X_\beta + \delta\cdot \beta(z) \cdot d\varphi^{T\cdot t} X_f\circ \varphi^{-T\cdot t},
\]
and its restriction to $Y_T$ is 
\[
X_{H_{\delta}}|_{Y_T} = T\cdot R +  \delta\cdot d\varphi^{T\cdot t} X_f\circ \varphi^{-T\cdot t}.
\]
\begin{lemma}
    For small enough $\delta>0$, the Hamiltonian $H_\delta$ has only nondegenerate 1-periodic orbits of action $T$, they all lie in $Y_T$ and are in bijection with critical points of $f|_{I\times Y_T}$:
    \[
    p\in \operatorname{Crit} (f|_{I\times Y_T})\quad \leftrightarrow \quad \gamma_p(t):= \varphi^{T\cdot t}(p)\subset Y_T 
    \]
     and 
    \[
    \indCZ(\gamma_p) = \indRS(S) -d + \ind_{\operatorname{Morse}}(f,p) = \mu_-(S)+ \ind_{\operatorname{Morse}}(f,p).
    \]
    where $S$ is the connected component of $S_T$ containing $\gamma_p$ and $d=\frac{1}{2}\dim (S)$
\end{lemma}
\begin{proof}
    The proof is similar to  \cite[Lemma 4.8]{chaidez2022contact}. 
    Outside of $\cN$, $H_\delta$ has no action-$T$ orbits, when $\delta$ is small enough. In $\cN$, the $H_\delta$ is simply the composition of $T\cdot \rho$ and $\delta\cdot f$. Therefore,
    \[
    \varphi^t_{H_\delta} = \varphi^{T\cdot t}\circ \varphi_{\delta f}^t.
    \]
    When $\delta$ is small, $\varphi^t_{\delta f}$ has no non-constant periodic orbits of period $T$ and hence $T$-periodic orbits of $H_\delta$ are in bijection with critical points of $f$ on $Y_T$. The linearized flow is given by the product $d\varphi^{T\cdot t}\circ \varphi_{\delta f}^t \cdot d\varphi_{\delta f}^t$.
    To see that $H_\delta$ is nondegenerate (for small $\delta$), recall that the Reeb flow on $\partial W$ is Morse--Bott, and therefore 
    \[
    d\varphi^{T}|_{V_T}=\operatorname{id},\quad \text{and}\quad  d\varphi^{T}|_{V_N} \quad \text{is nondegenerate.}
    \]
    On the other hand, 
     \[
    d\varphi_{\delta f}^T|_{V_N}=\operatorname{id},\quad \text{and}\quad  d\varphi_{\delta f}^T|_{V_T} \quad \text{is nondegenerate.}
    \]
    Finally, by the composition property of the Robbin--Salamon index \cite{g2014}, the Conley--Zehnder index of $\gamma_p$ is the sum of Robbin--Salamon indices of the two flows, and since $d\varphi_{\delta f}^t|_{V_T}$ crosses the Maslov cycle only at the identity, where the crossing form is minus the Hessian of $\delta f$, $\indRS(d\varphi_{\delta f}^t) = \ind_{\operatorname{Morse}}(f,p)-d$.
\end{proof}

The above lemma shows that the generators of the chain complex $CF^{[T-\epsilon,T+\epsilon]}(H_\delta)$ are the same as those of $CM(f)$ up to a grading shift by $\mu_-(S_T)$. Next, we need to prove that the differentials of these chain complexes agree. Since we work over $\F_2$ coefficients it is sufficient to construct a bijection between the relevant moduli spaces.

\begin{lemma}
    Let $J$ as in Setup~\ref{setup:J} and let $p,q\in Y_T$ be critical points of $f$. For small enough $\delta$, any Floer solution $u$ from $\gamma_p$ to $\gamma_q$ must be contained in $I\times Y_T\subset \R\times \partial W$.
\end{lemma}
\begin{proof}
    The fact that $\operatorname{image}(u)\subset \cN$ for small enough $\delta$ is standard and can be either seen from compactness or from arguments as in \cite[Proposition 3.2.]{hein2012conley}. It is also possible to carry the arguments below for $u^{-1}(\cN)$ instead of $u$, and deduce that $u\subset \cN$.
    
    Denote by $u_T=\pi\circ u$ the projection of $u$ to $I\times Y_T$. Note that $u_T$ is also a solution of the same Floer equation since $J$ respects the splitting $V_T\oplus V_N$ and commutes with  $d\pi$:
    \begin{align*}
    \partial_s u_T - J(\partial_t u_T -X_{H_\delta})&= d\pi\partial_s u - J(d\pi \partial_t u -X_{H_\delta})\\
    &= d\pi\partial_s u - J(d\pi \partial_t u -d\pi X_{H_\delta})\\
    &=d\pi\Big(\partial_s u - J(\partial_t u -X_{H_\delta})\Big) = d\pi(0).
    \end{align*}
    Here we used the fact that $X_{H_\delta}\in T(I\times Y_T)$ and hence $d\pi$ is the identity on it. Consider the energy of $u$, it splits as 
    \[
    E(u) = \int_{\R\times S^1} \|\partial_s u\|_J^2 \ ds\ dt = \int_{\R\times S^1} \|d\pi\partial_s u\|_J^2+ \|(\operatorname{Id}-d\pi)\partial_s u\|_J^2 \ ds\ dt =: E(u_T)+ E_N.
    \]
    However, by the energy identity of Floer solutions, since both $u$ and $u_T$ solve the same equation and have the same ends, their energies coincide: 
    \[
     E(u_T)+E_N = E(u)=\cA(\gamma_p) - \cA(\gamma_q) =E(u_T)\quad\text{which implies}\quad E_N=0.
    \]
    In particular $\partial_s u$ has no components in $\ker d\pi$.
    Note that $u$ is $J$-holomorphic in the $V_N$ directions: 
    $$(\operatorname{Id}-d\pi)\big(\partial_s u +J\partial_t u\big)=0.$$
    Therefore $\partial_t u$ has no components in $\ker d\pi$ as well and $u$ is constant in the normal directions. Since the ends lie in $Y_T$, we conclude that  $\operatorname{image}(u)\subset I\times Y_T$.
\end{proof}

\begin{lemma}
     Let $J$ as in Setup~\ref{setup:J} and let $p,q$ be critical points of $f$ of index difference $\leq 1$. For small enough $\delta$, there is a bijection between Floer solution $u\subset I\times Y_T$ from $\gamma_p$ to $\gamma_q$ and Morse negative-gradient flow lines of $\delta f$ from $p$ to $q$.
\end{lemma}
\begin{proof}
    We start with the   $C^2$-small, autonomous Morse Hamiltonian $\delta \cdot f$. For $C^2$ small autonomous Morse Hamiltonians, all Floer solutions are Morse flow lines (cf., \cite[Proposition 10.1.7.]{audin2014morse}). As explained in the setup above, the pair $(\delta f, g_J)$ is Morse-Smale, and the linearized operator of the Floer equation is surjective \cite[Section 10]{audin2014morse}. 
    
    The Reeb flow gives an $S^1$ action on $I\times Y_T$, so we can define the following map:
    \begin{equation}\label{eq:inverse_S1_action_Floer_sol}
        \Phi:u\mapsto\bar u, \quad \quad \bar u(s,t):= \varphi^{-T\cdot t} u(s,t).
    \end{equation}
    Then,  
    \begin{align*}
        \partial_s \bar u =& d\varphi^{-T\cdot t} \partial_s u,\\
        \partial_t \bar u =& -T\cdot R \circ \varphi^{-T\cdot t} u +d\varphi^{-T\cdot t} \partial_t u.
    \end{align*}
    Since $u$ satisfies the Floer equation for $H_{\delta} = T\cdot \rho +\delta\cdot f \circ \varphi^{-T\cdot t}$ in $\cN$, then $\bar u$ satisfies the Floer equation for the Hamiltonian 
    $\delta \cdot f $. Indeed,
    \begin{align*}
        \partial_s \bar u =& \ d\varphi^{-T\cdot t} \partial_s u = -d\varphi^{-T\cdot t} J(\partial_t u - X_{H_{\delta}})\\
        =& -Jd\varphi^{-T\cdot t}  \big(\partial_t u - T\cdot R - d\varphi^{T\cdot t} X_{\delta f} \circ\varphi^{-T\cdot t}\circ u\big)\\
        =& -J\big(d\varphi^{-T\cdot t}  \partial_t u - T\cdot R -  X_{\delta f} \circ\varphi^{-T\cdot t}\circ u\big)\\
        =& -J\big(\partial_t \bar u +T\cdot R - T\cdot R -X_{\delta f}\circ \bar u\big)\\
        =& -J\big(\partial_t \bar u -X_{\delta f}\circ \bar u\big).
    \end{align*}
    We see that the map (\ref{eq:inverse_S1_action_Floer_sol}) gives a bijection between Floer solutions of $H_\delta$ between $\gamma_q$ and $\gamma_p$ and those of $\delta\cdot f$ between $q,p$, which are simply Morse negative-gradient flow-lines.
\end{proof}
The fact that the linearized operator $L_{(H_\delta, J)}$ associated to $(H_\delta, J)$ is surjective at $u$ connecting $\gamma_q$ and $\gamma_p$ follows from the surjectivity of the linearized operator for $(\delta f, J)$, together with the fact that $L_{(\delta f,J)} = L_{(H_\delta, J)}\circ \Phi^{-1}$ where $\Phi$ is the map from (\ref{eq:inverse_S1_action_Floer_sol}).

\bibliography{refs}
\bibliographystyle{alpha}

\end{document}